\documentclass{amsart}
\usepackage{tabu}
\usepackage{amssymb} 
\usepackage{amsmath} 
\usepackage{amscd}
\usepackage{amsbsy}
\usepackage{comment, enumerate}
\usepackage[matrix,arrow]{xy}
\usepackage{hyperref}
\usepackage{mathrsfs}
\usepackage{color}
\usepackage{mathtools,caption}

\newcommand{\QQ}{\mathbb Q}
\newcommand{\F}{\mathbb F}
\newcommand{\ZZ}{\mathbb Z}
\newcommand{\OK}{\mathcal O_K}

\newcommand{\fp}{\mathfrak p}
\newcommand{\fq}{\mathfrak q}

\newcommand{\tu}{\tilde u}
\newcommand{\bound}{5,000}

\newenvironment{changemargin}[2]{%
\begin{list}{}{%
\setlength{\topsep}{0pt}%
\setlength{\leftmargin}{#1}%
\setlength{\rightmargin}{#2}%
\setlength{\listparindent}{\parindent}%
\setlength{\itemindent}{\parindent}%
\setlength{\parsep}{\parskip}%
}%
\item[]}{\end{list}}

\newcommand{\fz}{\mathfrak z}

\newcommand{\OL}{\mathcal O_L}

\newtheorem{thm}{Theorem}
\newtheorem{lem}{Lemma}[section]

\theoremstyle{definition}

\theoremstyle{remark}

\begin{document}

\title[]{Perfect powers that are sums of squares in a three term arithmetic progression}

\author{Angelos Koutsianas}
\address{Department of Informatics, University of Piraeus, Pireas, Greece}
\email{akoutsianas@webmail.unipi.gr}

\author{Vandita Patel}
\address{Department of Mathematics, University of Toronto, Bahen Centre, 40 St. George St., Room 6290, Toronto, Ontario, Canada, M5S 2E4}
\email{vandita@math.utoronto.ca}

\date{\today}

\keywords{Exponential equation, Lehmer sequences, primitive divisors}
\subjclass[2010]{Primary 11D61}

\begin{abstract}
We determine primitive solutions
to the equation $(x-r)^2 + x^2 + (x+r)^2 = y^n$ for $1 \le r \le \bound$, making
use of a factorization argument and the Primitive Divisors Theorem 
due to Bilu, Hanrot and Voutier.  
\end{abstract}

\maketitle

\section{Introduction} \label{intro}

Perfect powers that are sums of powers of consecutive terms 
in an arithmetic progression 
have attracted considerable attention.  
For example, Dickson's \emph{History of the Theory of Numbers} [Volume II,
582--588] \cite{Dickson} surveys the contributions of several prominent
mathematicians (including Cunningham, Catalan, Genocchi and Lucas) during the
early 19th and 20th century towards solving specific
cases of the Diophantine equation
\begin{equation}\label{eq:arithmetic_progressions}
x^k + (x + r)^k + \cdots +(x + (d-1)r)^k=y^n \quad x,y,d,k,r,n\in\ZZ, \; n\geq 2.
\end{equation}
This is still a remarkably active field, with recent results due
to
\cite{BaiZhang13}, \cite{Zhang14}, \cite{Hajdu15}, \cite{BennettPatelSiksek16}, \cite{BennettPatelSiksek17}, \cite{Patel17}, \cite{Soydan17}, \cite{BerczesPinkSavasSoydan18}, \cite{PatelSiksek}, \cite{Zhang17}, \cite{Koutsianas17b} and \cite{ArgaezPatel}.

In this paper, we consider the case $d=3$ and $k=2$, namely the equation
\begin{equation} \label{eq:main}
(x-r)^2 + x^2 + (x+r)^2 = y^n, \quad x,y,r,n \in \ZZ,\; n \ge 2.
\end{equation}
In  \cite{Koutsianas17b}, the first author studies equation~\eqref{eq:main}
where $r$ is of the form $p^b$ with $p$  a suitable prime. In this paper, we
completely solve \eqref{eq:main} for all values of $1\leq r\leq \bound$, under the
natural assumption $\gcd(x,y)=1$, using the characterization of primitive
divisors in Lehmer sequences due to Bilu, Hanrot and Voutier \cite{BiluHanrotVoutier01}.

An integer solution $(x,y)$ of \eqref{eq:main} is said to be
 \textit{primitive} if $\gcd(x,y)=1$. This is equivalent to $x,y,r$ being
pairwise coprime. A solution where $xy=0$ is called a \textit{trivial
solution}.

\begin{thm}\label{thm:main_gcd_one}
Let $1\leq r \leq \bound$. All non-trivial primitive solutions to equation \eqref{eq:main} 
with prime exponent $n$ are given in Table \ref{table:solutions_gcd_one}.
\end{thm}

\begin{changemargin}{-1.5cm}{-1.5cm}
\begin{center}
\begin{tabular}{| c | c | }
\hline
    $r$ &  $(|x|,y,n)$   \\ \hline
$ 2 $ & $ ( 21 , 11 , 3 ) $ \\ \hline
$ 7 $ & $ ( 3 , 5 , 3 ) $ \\ \hline
$ 11 $ & $ ( 31 , 5 , 5 ) $ \\ \hline
$ 70 $ & $ ( 862389 , 13067 , 3 ) $ \\ \hline
$ 79 $ & $ ( 63 , 29 , 3 ) $ \\ \hline
$ 92 $ & $ ( 93 , 35 , 3 ) $ \\ \hline
$ 119 $ & $ ( 801 , 125 , 3 ) $ \\ \hline
$ 133 $ & $ ( 17307 , 965 , 3 ) $ \\ \hline
$ 146 $ & $ ( 9 , 35 , 3 ) $ \\ \hline
$ 155 $ & $ ( 369 , 77 , 3 ) $ \\ \hline
$ 187 $ & $ ( 3255 , 317 , 3 ) $ \\ \hline
$ 196 $ & $ ( 207 , 59 , 3 ) $ \\ \hline
$ 197 $ & $ ( 13 , 5 , 7 ) $ \\ \hline
$ 205 $ & $ ( 147 , 53 , 3 ) $ \\ \hline
$ 223 $ & $ ( 345 , 77 , 3 ) $ \\ \hline
$ 262 $ & $ ( 89 , 11 , 5 ) $ \\ \hline
$ 371 $ & $ ( 374475 , 7493 , 3 ) $ \\ \hline
$ 376 $ & $ ( 1071 , 155 , 3 ) $ \\ \hline
$ 434 $ & $ ( 255 , 83 , 3 ) $ \\ \hline
$ 436 $ & $ ( 4169 , 35 , 5 ) $ \\ \hline
$ 439 $ & $ ( 987 , 149 , 3 ) $ \\ \hline
$ 623 $ & $ ( 291 , 101 , 3 ) $ \\ \hline
$ 713 $ & $ ( 30921 , 1421 , 3 ) $ \\ \hline
$ 727 $ & $ ( 2133 , 245 , 3 ) $ \\ \hline
$ 736 $ & $ ( 82035 , 2723 , 3 ) $ \\ \hline
$ 772 $ & $ ( 105 , 107 , 3 ) $ \\ \hline
$ 776 $ & $ ( 1545 , 203 , 3 ) $ \\ \hline
$ 866 $ & $ ( 861 , 155 , 3 ) $ \\ \hline
$ 889 $ & $ ( 1095 , 173 , 3 ) $ \\ \hline
\end{tabular}
\begin{tabular}{| c | c | }
\hline
    $r$ &  $(|x|,y,n)$   \\ \hline
$ 952 $ & $ ( 381 , 131 , 3 ) $ \\ \hline
$ 1087 $ & $ ( 3927 , 365 , 3 ) $ \\ \hline
$ 1136 $ & $ ( 9723 , 659 , 3 ) $ \\ \hline
$ 1190 $ & $ ( 1719 , 227 , 3 ) $ \\ \hline
$ 1316 $ & $ ( 54561 , 2075 , 3 ) $ \\ \hline
$ 1339 $ & $ ( 6069 , 485 , 3 ) $ \\ \hline
$ 1420 $ & $ ( 19413 , 1043 , 3 ) $ \\ \hline
$ 1442 $ & $ ( 1971 , 251 , 3 ) $ \\ \hline
$ 1469 $ & $ ( 4695 , 413 , 3 ) $ \\ \hline
$ 1519 $ & $ ( 6513 , 509 , 3 ) $ \\ \hline
$ 1636 $ & $ ( 357 , 179 , 3 ) $ \\ \hline
$ 1771 $ & $ ( 2097 , 269 , 3 ) $ \\ \hline
$ 1910 $ & $ ( 597 , 203 , 3 ) $ \\ \hline
$ 1955 $ & $ ( 21 , 197 , 3 ) $ \\ \hline
$ 1960 $ & $ ( 161823 , 4283 , 3 ) $ \\ \hline
$ 2009 $ & $ ( 294837 , 6389 , 3 ) $ \\ \hline
$ 2023 $ & $ ( 10035 , 677 , 3 ) $ \\ \hline
$ 2162 $ & $ ( 3729 , 371 , 3 ) $ \\ \hline
$ 2189 $ & $ ( 4053 , 389 , 3 ) $ \\ \hline
$ 2329 $ & $ ( 11109 , 725 , 3 ) $ \\ \hline
$ 2338 $ & $ ( 5505 , 467 , 3 ) $ \\ \hline
$ 2378 $ & $ ( 1651 , 11 , 7 ) $ \\ \hline
$ 2378 $ & $ ( 33808666101 , 15079691 , 3 ) $ \\ \hline
$ 2392 $ & $ ( 2826957 , 28835 , 3 ) $ \\ \hline
$ 2410 $ & $ ( 3171 , 347 , 3 ) $ \\ \hline
$ 2504 $ & $ ( 1659 , 275 , 3 ) $ \\ \hline
$ 2563 $ & $ ( 723 , 245 , 3 ) $ \\ \hline
$ 2567 $ & $ ( 13419 , 821 , 3 ) $ \\ \hline
$ 2599 $ & $ ( 14637 , 869 , 3 ) $ \\ \hline
\end{tabular}
\begin{tabular}{| c | c | }
\hline
    $r$ &  $(|x|,y,n)$   \\ \hline
$ 2788 $ & $ ( 1323 , 275 , 3 ) $ \\ \hline
$ 3026 $ & $ ( 6279 , 515 , 3 ) $ \\ \hline
$ 3098 $ & $ ( 3333 , 35 , 5 ) $ \\ \hline
$ 3109 $ & $ ( 627 , 29 , 5 ) $ \\ \hline
$ 3193 $ & $ ( 76365 , 2597 , 3 ) $ \\ \hline
$ 3247 $ & $ ( 20463 , 1085 , 3 ) $ \\ \hline
$ 3341 $ & $ ( 2961 , 365 , 3 ) $ \\ \hline
$ 3395 $ & $ ( 837 , 293 , 3 ) $ \\ \hline
$ 3472 $ & $ ( 11637 , 755 , 3 ) $ \\ \hline
$ 3859 $ & $ ( 29406489 , 137405 , 3 ) $ \\ \hline
$ 3967 $ & $ ( 27657 , 1325 , 3 ) $ \\ \hline
$ 4025 $ & $ ( 37257 , 1613 , 3 ) $ \\ \hline
$ 4034 $ & $ ( 9765 , 683 , 3 ) $ \\ \hline
$ 4228 $ & $ ( 2937 , 395 , 3 ) $ \\ \hline
$ 4268 $ & $ ( 216153 , 5195 , 3 ) $ \\ \hline
$ 4277 $ & $ ( 1011 , 341 , 3 ) $ \\ \hline
$ 4354 $ & $ ( 3447 , 419 , 3 ) $ \\ \hline
$ 4417 $ & $ ( 459 , 341 , 3 ) $ \\ \hline
$ 4529 $ & $ ( 680936595 , 1116293 , 3 ) $ \\ \hline
$ 4592 $ & $ ( 7305 , 587 , 3 ) $ \\ \hline
$ 4633 $ & $ ( 171057 , 4445 , 3 ) $ \\ \hline
$ 4669 $ & $ ( 59007 , 2189 , 3 ) $ \\ \hline
$ 4687 $ & $ ( 1277 , 5 , 11 ) $ \\ \hline
$ 4712 $ & $ ( 1530639 , 371 , 5 ) $ \\ \hline
$ 4718 $ & $ ( 8397 , 635 , 3 ) $ \\ \hline
$ 4759 $ & $ ( 36363 , 1589 , 3 ) $ \\ \hline
$ 4808 $ & $ ( 1269 , 371 , 3 ) $ \\ \hline
$ 4961 $ & $ ( 4451643 , 39029 , 3 ) $ \\ \hline
 		 & 								\\ \hline
\end{tabular}
\captionof{table}{Triples of non-trivial primitive solutions $(|x|,y,n)$ 
\\
of \eqref{eq:main} for the values of $1 \leq r\leq \bound$.}
\label{table:solutions_gcd_one}
\end{center}
\end{changemargin}

{\bf Remark:}
Non-trivial primitive solutions to equation \eqref{eq:main}, with $1\leq r\leq \bound$  
and exponent $n$ that is composite can also be recovered from Table \ref{table:solutions_gcd_one}. This can be done by simply checking whether $y$ is a perfect power.

\section{Prime divisors of Lehmer sequences}
A \textit{Lehmer pair} is a pair $\alpha,\beta$ of algebraic integers such that $(\alpha +\beta)^2$ and $\alpha\beta$ are non--zero coprime rational integers and $\alpha/\beta$ is not a root of unity. The \textit{Lehmer sequence} 
associated to the Lehmer pair $(\alpha,\beta)$ is
\begin{equation}
\tu_n=\tu_n(\alpha,\beta)=\begin{cases}
\frac{\alpha^n-\beta^n}{\alpha-\beta} & n\text{ odd},\\
\frac{\alpha^n-\beta^n}{\alpha^2-\beta^2} & n\text{ even}.
\end{cases}
\end{equation}
A prime $p$ is called a \textit{primitive divisor} of $\tu_n$ if
it divides $\tu_n$ but does not divide
$(\alpha^2-\beta^2)^2\cdot\tu_1\cdots\tu_{n-1}$. 
We shall make use of the following celebrated theorem~\cite{BiluHanrotVoutier01}.
\begin{thm}[Bilu, Hanrot and Voutier]\label{thm:non_defective}
Let $\alpha$, $\beta$ be a Lehmer pair. Then
$\tu_n(\alpha,\beta)$ has a primitive divisor for all $n>30$,
and for all prime $n>7$.
\end{thm}

\section{Proof of Theorem \ref{thm:main_gcd_one}}
We can rewrite \eqref{eq:main} as  
\begin{equation}\label{eq:main_simply}
3x^2+2r^2=y^n.
\end{equation}
Suppose $\gcd(x,y)=1$; this implies that $x$, $y$, $r$
are pairwise coprime.
Note that $n\ne 2$ as $2$ is a quadratic non-residue modulo $3$.
We shall henceforth suppose that $n$ is an odd prime. 
We rewrite  \eqref{eq:main_simply} as
\begin{equation}\label{eq:3squares}
(3x)^2 + 6r^2 = 3y^n,
\end{equation}
Let $K = \QQ(\sqrt{-6})$ and write $\OK=\ZZ[\sqrt{-6}]$ for its ring of integers.
This has class group isomorphic to $\ZZ/2\ZZ$. We factorize the left-hand side
of equation~\eqref{eq:3squares} as
$$
(3x + r\sqrt{-6})(3x - r\sqrt{-6}) = 3y^n.
$$
It follows that
\begin{equation}\label{eq:ideal_factorization}
(3x + r\sqrt{-6})\OK = \fp_3 \cdot \fz^n
\end{equation}
where $\fp_3$ is the unique prime of $\OK$ above $3$ and $\fz$ is an ideal of $\OK$. The ideal $\fp_3$ is not principal, thus $\fz$ is not either, and $\fp_3^2=(3)$. We write
$$
(3x + r\sqrt{-6})\OK = \fp_3^{1-n} \cdot (\fp_3 \fz)^n
=(3^{(1-n)/2})(\fp_3 \fz)^n.
$$
It follows that $\fp_3 \fz$ is a principal ideal. Write $\fp_3\fz=(\gamma) \OK$ where $\gamma=u+v\sqrt{-6} \in \OK$ with $u,v\in\ZZ$. After possibly changing the sign of $\gamma$ we obtain,
\begin{equation}\label{eq:x_r_gamma}
3x+r\sqrt{-6}=\frac{\gamma^n}{3^{(n-1)/2}}.
\end{equation}
Subtracting the conjugate equation from this equation, we obtain 
\begin{equation}\label{eqn:thue}
\frac{\gamma^n}{3^{(n-1)/2}} - \frac{\bar{\gamma}^n}{3^{(n-1)/2}} = 2\cdot r\sqrt{-6},
\end{equation}
or equivalently,
\begin{equation}\label{eq:Lehmer_sequence}
\frac{\gamma^n}{3^{n/2}} - \frac{\bar{\gamma}^n}{3^{n/2}} = 2\cdot r\sqrt{-2}.
\end{equation}

Let $L = \QQ(\sqrt{-6}, \sqrt{3}) =  \QQ(\sqrt{-2}, \sqrt{3}) $. Write $\OL$ for the ring of integers of $L$ and let
$$
\alpha = \frac{\gamma}{\sqrt{3}} \qquad \text{and} \qquad 
\beta = \frac{\bar{\gamma}}{\sqrt{3}}.
$$

\begin{lem}\label{lem:lehmer}
Let $\alpha, \beta$ be as above. Then, $\alpha$ and $\beta$ are algebraic integers. Moreover, $(\alpha+\beta)^2$ and $\alpha\beta$ are non--zero coprime rational integers and $\alpha/\beta$ is not a unit. 
\end{lem}

\begin{proof}
Let $\gamma=u+v\sqrt{-6}$ with $u,v\in\ZZ$. Then
$$(\alpha+\beta)^2=\frac{4u^2}{3}.$$
Since $\fp_3\fz=(\gamma)\OK$ and $\fp_3\mid \sqrt{-6}$ we conclude that $\fp_3\mid u$ and so $3\mid u$. So, $(\alpha+\beta)^2$ is a rational integer. If $(\alpha+\beta)^2=0$ then we have $u=0$. However, from \eqref{eq:Lehmer_sequence} and the fact that $n$ is odd we understand that this cannot happen. Clearly, $\alpha\beta=\gamma\bar\gamma/3$ is a non--zero rational integer.

We have to check that $(\alpha+\beta)^2$ and $\alpha\beta$ are coprime. Suppose they are not coprime. Then there exist a prime $\fq$ of $\OL$ dividing both. Then $\fq$ divides $\alpha,\beta$ and from equations \eqref{eq:x_r_gamma} and \eqref{eq:Lehmer_sequence} we understand that $\fq$ divides $(y)\OL$ and $(2r\sqrt{-2})\OL$ which contradicts the assumption that $(x,y)$ is a non--trivial primitive solution. 

Finally, we need to show that $\alpha/\beta=\gamma/\bar\gamma\in\OK$ is not a root of unity. Since the only roots of unity in $K$ are $\pm 1$ we conclude $\gamma=\pm\bar\gamma$. Then, either $v=0$ or $u=0$ which both cannot hold because of \eqref{eq:Lehmer_sequence}.
\end{proof}

From Lemma \ref{lem:lehmer} we have that the pair $(\alpha,\beta)$ is Lehmer pair and we denote by $\tu_k$ the associate Lehmer sequence. Substituting into equation~\eqref{eq:Lehmer_sequence}, we see that
\begin{equation}
\left(\frac{\alpha-\beta}{2\sqrt{-2}} \right)\left(\frac{\alpha^n - \beta^n}{\alpha-\beta}\right) = r .
\end{equation}

Hence, we get:
\begin{equation}\label{eqn:rprime}
\frac{\alpha^n - \beta^n}{\alpha-\beta}=  r /v = r^{\prime}.
\end{equation}



\begin{lem}\label{lem:B}
For a prime $q \nmid 6$, let
\[
B_q=\begin{cases}
q-1 & \text{if $\left(\frac{-6}{q}\right)=1$}\\
q+1 & \text{if $\left(\frac{-6}{q}\right)=-1$}.
\end{cases}
\]
Let
\[
B:=\max\left(
\{7\} \cup \{
B_q \; : \; \text{$q$ prime, $q\mid r^\prime$, $q \nmid 6v$}
\}
\right).
\]
Then $n\le B$.
\end{lem}

\begin{proof}
Recall that the exponent $n$ is an odd prime. Suppose $n>7$.
By the theorem of Bilu, Hanrot and Voutier,
$\tilde{u}_n=(\alpha^n-\beta^n)/(\alpha-\beta)=r^\prime$
is divisible by a prime $q$ that does not divide 
$(\alpha^2-\beta^2)^2=-32 u^2 v^2/3$ nor the terms
$\tilde{u}_1,\tilde{u}_2,\dotsc,\tilde{u}_{n-1}$.
Note that this is a prime $q$ dividing $r^\prime$
but not $6v$. Let $\fq$ be a prime of $K=\QQ(\sqrt{-6})$
above $q$. 
As
$(\alpha+\beta)^2$ and $\alpha\beta$ are coprime integers,
and as $\alpha$, $\beta$ satisfy \eqref{eqn:rprime} we see that
$\gamma$, $\overline{\gamma}$
are not divisible by $\fq$. We claim the multiplicative order
of the reduction of $\gamma/\overline{\gamma}$ modulo $\F_\fq$
divides $B_q$. If $-6$ is a square modulo $q$, then $\F_\fq=\F_q$
and so the multiplicative order divides $q-1=B_q$.
Otherwise, $\F_\fq=\F_{q^2}$. However, $\gamma/\overline{\gamma}$
has norm $1$, and the elements of norm $1$ in $\F_{q^2}^*$
form a subgroup of order $q+1=B_q$. Thus in either case
\[
(\gamma/\overline{\gamma})^{B_q} \equiv 1 \pmod{\fq}
\]
This implies that $q \mid \tilde{u}_{B_q}$. As $q$
is primitive divisor of $\tilde{u}_{n}$ we see that $n \le B_q$,
proving the lemma.
\end{proof}

\begin{proof}[Proof of Theorem \ref{thm:main_gcd_one}]
We notice that since $\gcd(x,r)=1$ we can immediately
deduce that $3\nmid r$.
We wrote a simple \texttt{Sage} \cite{Sage} script which for
each $1 \le r \le \bound$ such that  $3 \nmid r$, and for each $v \mid r$ computed
$B$ as in Lemma~\ref{lem:B}. For each odd prime $n \le B$
we know from equation~\eqref{eqn:thue} that $u$ is an integer solution of the polynomial equation


\begin{equation}\label{eqn:solve_fixed_rprime}
\frac{1}{2 \cdot r \cdot \sqrt{-6} \cdot 3^{(n-1)/2}}
\cdot \left( (u+v \sqrt{-6})^n-(u-v\sqrt{-6})^n\right) \; -1
\end{equation}

Computing the roots of these polynomials we are able to obtain the solutions $(|x|,y,n)$ as in Table~\ref{table:solutions_gcd_one}.
\end{proof}


\vspace*{0.5cm}

\section*{Acknowledgement}
Both authors would like to thank Professor John Cremona for providing access to the servers of the Number Theory Group of Warwick Mathematics Institute where all the computations took place.

The first author would like to thank
Institut f\"{u}r Reine Mathematik in Universit\"{a}t Ulm for the nice working
environment because a part of the work took place when he was a member of the
institute.

The second author is incredibly indebted to the University of Toronto and their mathematics department for their amazing support during such a difficult period. Special thanks have to go to Professor Dror Bar-Natan, Professor Nick Hoell, Professor Kumar Murty and Patrina Seepersaud for their kindness and support. 
Gratitude must be extended to all members of GANITA Lab who helped in one way or another: Abhishek, Anup, Debanjana, Gaurav, Payman, Ren, Samprit, Shenhui, Shuyang   and last but definitely not least, Professor Pramath Sastry.
A very special mention goes to Amit and Anja because frankly, without their help, things would have definitely been infinitely worse.
Finally, she would like to thank the members of the Number Theory community, for their generosity and well wishes, especially to Debanjana Kundu, Professor Pieter Moree , Gaurav Patil and Professor Samir Siksek, who all brought the math virtually to her when she could not physically get to it and filled her home with number theory.

\bibliographystyle{alpha}
\bibliography{squares_in_AP2}

\end{document}